\documentclass[11pt]{amsart}
\usepackage[a4paper,margin=3.1cm]{geometry}
\usepackage{amsmath,amssymb,amsthm}
\usepackage{graphicx}
\usepackage[hidelinks]{hyperref}
\newtheorem{theorem}{Theorem}
\newtheorem{lemma}[theorem]{Lemma}
\newtheorem{proposition}[theorem]{Proposition}
\newcommand{\R}{\mathbb{R}}
\newcommand{\dimfty}{\dim_{\infty}}
\newcommand{\ceil}[1]{\left\lceil #1\right\rceil}
\newcommand{\floor}[1]{\left\lfloor #1\right\rfloor}
\title[Isometric embeddings of trees]{Isometric embeddings of trees: why 32 leaves can require six dimensions, and bounds for every tree}
\author{Logan R.\ Chalmers}
\address{University of Otago, Dunedin 9016, New Zealand}
\email{logan.chalmers@postgrad.otago.ac.nz}
\date{}
\begin{document}
\begin{abstract}
We identify the first failure of the $\ceil{\log_2t}$ leaf-count bound
for isometric embeddings of $t$-leaf trees under the maximum norm.
Every tree with at most $31$ leaves attains the bound, but a $32$-leaf
tree requires six dimensions. This answers a question of Fitzpatrick
and Nowakowski and disproves a conjecture of Aksoy, K{\i}l{\i}\c{c}
and Ko\c{c}ak.
We explain this failure by showing how requirements within and between
branches force an additional coordinate. For a family of three-branch
trees, we prove that the leaf bound is attained exactly when the two
larger branches have equal leaf
counts; otherwise, one additional coordinate is necessary. We also
prove that every tree embeds in $\ceil{1.162832\log_2t}$ dimensions,
improving the coefficient of approximately $1.44042$ established in
2005. These results hold for arbitrary positive edge lengths and the
entire metric realization.
We also construct infinite families in which the excess persists under
repeated leaf expansion.
Both main theorems are formalised in Lean using only standard axioms.
\end{abstract}
\maketitle
\section{Introduction}

How many coordinates are needed to represent a tree without changing
its distances? We measure distance between coordinate vectors by the
largest difference in any one coordinate. An \emph{isometric embedding}
assigns coordinates to the vertices so that this distance agrees with
the length of the path between them in the tree.

Let $T$ be a finite tree with at least two vertices, and let $d_T(u,v)$
denote the number of edges on its unique $u$--$v$ path. The isometric
$\ell_\infty$-dimension of $T$, denoted by $\dimfty(T)$, is the least
integer $m$ for which a map $f:V(T)\to\R^m$ satisfies
\[
 d_T(u,v)=\max_{1\le i\le m}|f_i(u)-f_i(v)|
 \qquad(u,v\in V(T)).
\]
For a tree with $t$ leaves, Fitzpatrick and Nowakowski
\cite{FitzpatrickNowakowski2000} proved
\begin{equation}\label{eq:leaf-bound}
 \dimfty(T)\ge\ceil{\log_2t}.
\end{equation}
For a star, this many coordinates suffice. Set $m=\ceil{\log_2t}$,
map the centre to the origin, and map the leaves to distinct vectors
in $\{-1,1\}^m$. Each vector has maximum-norm distance one from the
origin, and any two distinct vectors have distance two, as required
by the star's path metric.

Does the number of leaves determine the isometric dimension of every
tree, even when the branching is spread across several vertices?
The bound is also exact for caterpillars. Fitzpatrick and
Nowakowski asked in 2000 whether this bound can be strict
\cite[Problem~43]{FitzpatrickNowakowski2000}; Brigham, Chartrand,
Dutton and Zhang proved in 2005 that equality holds for every tree with
at most $21$ leaves \cite[p.~281]{BrighamChartrandDuttonZhang2005}.
In 2008, Peterin announced equality for all trees in a conference
abstract, without a proof \cite[p.~19]{Peterin2008}.

We answer Fitzpatrick and Nowakowski's question by locating the first
failure. Every tree with at most $31$ leaves attains the leaf bound,
but a $32$-leaf tree requires six dimensions. This is incompatible
with the equality announced by Peterin in 2008. We give structural
proofs of both the obstruction and the sharp threshold.

\begin{theorem}\label{thm:first}
Every tree with $2\le t\le31$ leaves satisfies
\[
 \dimfty(T)=\ceil{\log_2t}.
\]
Every $32$-leaf tree has dimension at most six, and there is a
$32$-leaf tree, with all vertices of degree one or three, whose
dimension is six.
\end{theorem}

To explain the failure, we examine how branches share coordinates.
The same coordinates must preserve distances both within each branch
and between different branches. We determine when these requirements
force an extra coordinate in a family of three-branch trees.
The leaf bound is attained
exactly when the two larger branches have equal leaf counts.
Unequal counts force one additional coordinate.
For every $k\ge5$, we construct trees with $2^k$ leaves
and dimension $k+1$.
Repeatedly attaching two new leaves to each existing leaf eventually
stabilizes the excess over the leaf bound. We construct families in which it
disappears after one step and others in which it persists.

Changing the edge lengths cannot remove the obstruction.
For positive edge lengths, the \emph{metric realization} replaces
each edge by an interval of its assigned length, with distance given
by path length. Aksoy, K{\i}l{\i}\c{c} and Ko\c{c}ak
\cite[Introduction]{AksoyKilicKocak2020} conjectured in 2020 that every
such metric tree with at most $2^m$ leaves embeds isometrically into
$\ell_\infty^m$. They proved necessity of the leaf condition and
sufficiency for stars. Our $32$-leaf example has dimension six for
every positive edge weighting, both on its vertices and on its entire
metric realization, disproving this conjecture at $m=5$.

Once the leaf bound fails, how many dimensions suffice for an arbitrary tree?
Following the $O(\log t)$ bound of Linial, London and Rabinovich
\cite[Theorem~5.3]{LinialLondonRabinovich1995}, Fitzpatrick and
Nowakowski proved the upper bound $2\ceil{\log_2t}$.
In 2005, Brigham, Chartrand, Dutton and Zhang improved it to
$\ceil{1.440420091\log_2t}$ and asked for the best logarithmic
coefficient \cite[Theorem~6 and p.~280]{BrighamChartrandDuttonZhang2005}.
We reduce this coefficient to $1.162832$.

\begin{theorem}\label{thm:universal}
Let
\[
 \alpha=\frac{\log 2}{\log(363/200)}=1.162831872\ldots.
\]
Every tree with $t\ge2$ leaves embeds isometrically into
$\ell_\infty^{\ceil{\alpha\log_2t}}$. The same bound holds for the
entire metric realization under arbitrary positive edge lengths.
\end{theorem}

Both theorems are formalised in Lean using only its standard axioms
\cite{FormalProofs}.

\section{Preliminaries}\label{sec:interfaces}

An orientation of a tree covers a pair of leaves when their entire joining
path is directed, in either direction. An \emph{orientation cover} is a
family covering every pair of distinct leaves. Write $D(T)$ for its minimum
size. The following equivalence reduces the embedding problem to orientation covers.
The unweighted equivalence appears in
\cite[Corollary~26]{FitzpatrickNowakowski2000} and
\cite[Lemma~3]{BrighamChartrandDuttonZhang2005}.

\begin{proposition}\label{prop:metric}
For every positive edge weighting of a finite tree $T$, the isometric
$\ell_\infty$-dimensions of its vertex metric and its full metric realization
both equal $D(T)$. In particular, these dimensions do not depend on the
positive edge lengths.
\end{proposition}
\begin{proof}
Given $m$ covering orientations, choose a base vertex and integrate signed
edge length in each orientation. Extend each coordinate affinely along
every edge. Each coordinate is $1$-Lipschitz. Every segment in the metric
tree extends to a leaf-to-leaf path; an orientation directing that path
attains the full distance on the segment. The resulting map into
$\ell_\infty^m$ is therefore isometric.

Conversely, suppose the vertices embed isometrically into $\ell_\infty^m$.
Orient an edge in coordinate $i$ according to the sign of its coordinate
difference, breaking ties arbitrarily. For each leaf pair, some coordinate
attains the full path length. Equality in the sum of the edgewise
$1$-Lipschitz inequalities forces every difference on that path to have
its full edge length and the same sign. The path is directed in that
coordinate. This gives an orientation cover.
\end{proof}

The \emph{signature} of a leaf records the directions of its pendant edge,
with $0$ meaning toward the leaf and $1$ away from it. Leaves with the same
signature cannot be joined by a directed path. Thus an $m$-cover has at
most $2^m$ leaves, proving \eqref{eq:leaf-bound}.

Distinct leaf signatures do not by themselves ensure that any coordinate
directs the entire joining path. The intervening edges impose additional
compatibility requirements.

Suppressing degree-two vertices does not change $D(T)$. For upper bounds
we may also refine a vertex of higher degree into a trivalent tree without
changing the leaves. A directed path in the refinement remains directed
after the added edges are contracted. We shall consequently prove upper
bounds for trees whose vertices have degree one or three.

\subsection{Branch interfaces}
To combine covers of different branches, we must record which paths
from their attachment vertex are directed in each coordinate.
An \emph{attached branch} has an external attachment vertex, a single
attachment edge, and a rooted binary tree beyond it. The single-edge
branch is denoted $1$, and $(A,B)$ denotes the branch obtained by joining
two child branches and adding their parent attachment edge.
An internal cover need only cover paths between ordinary leaves, not
paths from the attachment vertex.

Put $Q_m=\{0,1\}^m$ and identify its vertices with binary integers.
Write $I_n=\{0,\ldots,n-1\}$, padding with fixed zero coordinates
when a larger ambient cube is needed.
Relative to the attachment vertex, encode an outward edge by $0$ and
an inward edge by $1$. For an internal $m$-cover of a branch
$B$, let $a\in Q_m$ be its attachment-edge word. For each leaf $x$, let
$M_x$ be the coordinates directing the entire attachment-to-$x$ path.
Define
\begin{equation}\label{eq:interface}
 F_x=\{z\in Q_m:z_i=a_i\text{ for all }i\in M_x\},
 \qquad U(B)=\bigcup_x F_x.
\end{equation}
We call $U(B)$ the \emph{interface}. It records exactly which directions
of a new pendant edge would fail to cover a path to some leaf of $B$.
An empty $M_x$ gives $F_x=Q_m$. An interface is \emph{tight} if its
cardinality equals the number of leaves in the branch.

A set is \emph{unate toward} $a$ if replacing any coordinate by the
corresponding bit of $a$ keeps its points in the set. Equivalently,
reflection in the coordinates where $a_i=1$ makes it a down-set.
Write $\downarrow_a S$ for the smallest such set containing $S$.
Every interface in \eqref{eq:interface} is unate toward its attachment word.

\begin{lemma}[Interface rules]\label{lem:interfaces}
The following rules describe how interfaces combine.
\begin{enumerate}
\item Branches at a common attachment vertex cover all cross-branch
leaf pairs if and only if their interfaces are pairwise disjoint.
\item Coordinate permutations and reflections of an interface are
implemented by permuting and reversing whole orientations.
Adding a radial coordinate replaces $U$ by $U\times\{0\}$.
\item If child interfaces $U_1,U_2$ are disjoint, prepending an edge
with word $a$ gives the parent interface
$\downarrow_a(U_1\cup U_2)$.
\item An internally covered branch with $n$ leaves has $|U|\ge n$.
If equality holds at a parent, it holds at both children and the parent
interface is their disjoint union.
\end{enumerate}
\end{lemma}
\begin{proof}
Two leaf cylinders are disjoint exactly when they fix opposite bits in
some common coordinate. This is exactly a coordinate directing the two
radial paths oppositely, and hence their whole joining path. This proves
the first assertion; the second follows directly from the definitions.
For the third, the new edge retains a fixed coordinate of a child
cylinder precisely when its bit agrees with $a$. Freeing the other fixed
coordinates gives its hull toward $a$. Hulls commute with unions.

Every pendant signature belongs to its leaf cylinder, and internal
coverage makes these signatures distinct. Hence $|U|\ge n$.
At a tight parent,
$n_1+n_2=|U|\ge |U_1|+|U_2|\ge n_1+n_2$;
all inequalities are equalities.
\end{proof}

At a junction with $2^m$ leaves, an $m$-cover therefore partitions $Q_m$
into its branch interfaces. Tightness propagates through every descendant
branch. Thus equality in the leaf bound requires these unate regions to fit
without any overlap or enlargement.

\subsection{Perfect branches}
A \emph{face} is obtained by fixing some cube coordinates and leaving the
others free. Let $F_r$ be the perfect binary branch with $2^r$ leaves.

\begin{lemma}\label{lem:balanced-face}
Every interface of $F_r$ contains an $r$-face through its attachment
word. In particular, a tight interface is an $r$-face. The branch also
attains such an interface in $r$ coordinates.
\end{lemma}
\begin{proof}
The assertion for $r=0$ is immediate. At the next fork, each child
interface contains an $r$-face through its own attachment word, by
induction. Take their hulls toward the parent word $a$.
If either face has a fixed bit different from $a$, its hull contains an
$(r+1)$-face through $a$. If neither does, both child faces contain $a$,
contradicting disjointness. This proves the containment. Tightness forces
equality by cardinality. For attainability, recursively allocate opposite
faces to the two perfect children and take a corner of their union as
the attachment word.
\end{proof}

\section{A 32-leaf tree requiring six dimensions}\label{sec:original}
The tree $S$ has three branches with $4$, $13$ and $15$ leaves
(Figure~\ref{fig:original}). To specify their shapes, set
\[
 E=(1,1),\quad P=(1,E),\quad C_4=(1,P),\quad
 C_5=(1,C_4),\quad F=(E,E).
\]
Let $S$ have the three branches
\begin{equation}\label{eq:original}
 F,\qquad A=(E,(P,(E,(E,F)))),\qquad
 B=(E,(C_5,(C_4,F)))
\end{equation}
at a common vertex. Every vertex has degree one or three.

\begin{figure}[ht]
\centering
\setlength{\unitlength}{1mm}
\resizebox{\linewidth}{!}{%
\begin{picture}(130,47)
\qbezier(63.000,42.320)(35.500,38.920)(8.000,35.520)
\qbezier(8.000,35.520)(6.000,32.800)(4.000,30.080)
\qbezier(4.000,30.080)(3.000,27.020)(2.000,23.960)
\qbezier(4.000,30.080)(5.000,27.020)(6.000,23.960)
\qbezier(8.000,35.520)(10.000,32.800)(12.000,30.080)
\qbezier(12.000,30.080)(11.000,27.020)(10.000,23.960)
\qbezier(12.000,30.080)(13.000,27.020)(14.000,23.960)
\qbezier(63.000,42.320)(46.000,38.920)(29.000,35.520)
\qbezier(29.000,35.520)(24.500,32.800)(20.000,30.080)
\qbezier(20.000,30.080)(19.000,27.020)(18.000,23.960)
\qbezier(20.000,30.080)(21.000,27.020)(22.000,23.960)
\qbezier(29.000,35.520)(33.500,32.800)(38.000,30.080)
\qbezier(38.000,30.080)(33.500,27.360)(29.000,24.640)
\qbezier(29.000,24.640)(27.500,21.580)(26.000,18.520)
\qbezier(29.000,24.640)(30.500,21.920)(32.000,19.200)
\qbezier(32.000,19.200)(31.000,16.140)(30.000,13.080)
\qbezier(32.000,19.200)(33.000,16.140)(34.000,13.080)
\qbezier(38.000,30.080)(42.500,27.360)(47.000,24.640)
\qbezier(47.000,24.640)(43.500,21.920)(40.000,19.200)
\qbezier(40.000,19.200)(39.000,16.140)(38.000,13.080)
\qbezier(40.000,19.200)(41.000,16.140)(42.000,13.080)
\qbezier(47.000,24.640)(50.500,21.920)(54.000,19.200)
\qbezier(54.000,19.200)(51.000,16.480)(48.000,13.760)
\qbezier(48.000,13.760)(47.000,10.700)(46.000,7.640)
\qbezier(48.000,13.760)(49.000,10.700)(50.000,7.640)
\qbezier(54.000,19.200)(57.000,16.480)(60.000,13.760)
\qbezier(60.000,13.760)(58.000,11.040)(56.000,8.320)
\qbezier(56.000,8.320)(55.000,5.260)(54.000,2.200)
\qbezier(56.000,8.320)(57.000,5.260)(58.000,2.200)
\qbezier(60.000,13.760)(62.000,11.040)(64.000,8.320)
\qbezier(64.000,8.320)(63.000,5.260)(62.000,2.200)
\qbezier(64.000,8.320)(65.000,5.260)(66.000,2.200)
\qbezier(63.000,42.320)(73.562,38.920)(84.125,35.520)
\qbezier(84.125,35.520)(78.062,32.800)(72.000,30.080)
\qbezier(72.000,30.080)(71.000,27.020)(70.000,23.960)
\qbezier(72.000,30.080)(73.000,27.020)(74.000,23.960)
\qbezier(84.125,35.520)(90.188,32.800)(96.250,30.080)
\qbezier(96.250,30.080)(89.000,27.360)(81.750,24.640)
\qbezier(81.750,24.640)(79.875,21.580)(78.000,18.520)
\qbezier(81.750,24.640)(83.625,21.920)(85.500,19.200)
\qbezier(85.500,19.200)(83.750,16.140)(82.000,13.080)
\qbezier(85.500,19.200)(87.250,16.480)(89.000,13.760)
\qbezier(89.000,13.760)(87.500,10.700)(86.000,7.640)
\qbezier(89.000,13.760)(90.500,11.040)(92.000,8.320)
\qbezier(92.000,8.320)(91.000,5.260)(90.000,2.200)
\qbezier(92.000,8.320)(93.000,5.260)(94.000,2.200)
\qbezier(96.250,30.080)(103.500,27.360)(110.750,24.640)
\qbezier(110.750,24.640)(106.125,21.920)(101.500,19.200)
\qbezier(101.500,19.200)(99.750,16.140)(98.000,13.080)
\qbezier(101.500,19.200)(103.250,16.480)(105.000,13.760)
\qbezier(105.000,13.760)(103.500,10.700)(102.000,7.640)
\qbezier(105.000,13.760)(106.500,11.040)(108.000,8.320)
\qbezier(108.000,8.320)(107.000,5.260)(106.000,2.200)
\qbezier(108.000,8.320)(109.000,5.260)(110.000,2.200)
\qbezier(110.750,24.640)(115.375,21.920)(120.000,19.200)
\qbezier(120.000,19.200)(118.000,16.480)(116.000,13.760)
\qbezier(116.000,13.760)(115.000,10.700)(114.000,7.640)
\qbezier(116.000,13.760)(117.000,10.700)(118.000,7.640)
\qbezier(120.000,19.200)(122.000,16.480)(124.000,13.760)
\qbezier(124.000,13.760)(123.000,10.700)(122.000,7.640)
\qbezier(124.000,13.760)(125.000,10.700)(126.000,7.640)
\put(63.000,42.320){\makebox(0,0){\tiny$\bullet$}}
\put(8.000,35.520){\makebox(0,0){\tiny$\bullet$}}
\put(4.000,30.080){\makebox(0,0){\tiny$\bullet$}}
\put(2.000,23.960){\makebox(0,0){\tiny$\bullet$}}
\put(6.000,23.960){\makebox(0,0){\tiny$\bullet$}}
\put(12.000,30.080){\makebox(0,0){\tiny$\bullet$}}
\put(10.000,23.960){\makebox(0,0){\tiny$\bullet$}}
\put(14.000,23.960){\makebox(0,0){\tiny$\bullet$}}
\put(29.000,35.520){\makebox(0,0){\tiny$\bullet$}}
\put(20.000,30.080){\makebox(0,0){\tiny$\bullet$}}
\put(18.000,23.960){\makebox(0,0){\tiny$\bullet$}}
\put(22.000,23.960){\makebox(0,0){\tiny$\bullet$}}
\put(38.000,30.080){\makebox(0,0){\tiny$\bullet$}}
\put(29.000,24.640){\makebox(0,0){\tiny$\bullet$}}
\put(26.000,18.520){\makebox(0,0){\tiny$\bullet$}}
\put(32.000,19.200){\makebox(0,0){\tiny$\bullet$}}
\put(30.000,13.080){\makebox(0,0){\tiny$\bullet$}}
\put(34.000,13.080){\makebox(0,0){\tiny$\bullet$}}
\put(47.000,24.640){\makebox(0,0){\tiny$\bullet$}}
\put(40.000,19.200){\makebox(0,0){\tiny$\bullet$}}
\put(38.000,13.080){\makebox(0,0){\tiny$\bullet$}}
\put(42.000,13.080){\makebox(0,0){\tiny$\bullet$}}
\put(54.000,19.200){\makebox(0,0){\tiny$\bullet$}}
\put(48.000,13.760){\makebox(0,0){\tiny$\bullet$}}
\put(46.000,7.640){\makebox(0,0){\tiny$\bullet$}}
\put(50.000,7.640){\makebox(0,0){\tiny$\bullet$}}
\put(60.000,13.760){\makebox(0,0){\tiny$\bullet$}}
\put(56.000,8.320){\makebox(0,0){\tiny$\bullet$}}
\put(54.000,2.200){\makebox(0,0){\tiny$\bullet$}}
\put(58.000,2.200){\makebox(0,0){\tiny$\bullet$}}
\put(64.000,8.320){\makebox(0,0){\tiny$\bullet$}}
\put(62.000,2.200){\makebox(0,0){\tiny$\bullet$}}
\put(66.000,2.200){\makebox(0,0){\tiny$\bullet$}}
\put(84.125,35.520){\makebox(0,0){\tiny$\bullet$}}
\put(72.000,30.080){\makebox(0,0){\tiny$\bullet$}}
\put(70.000,23.960){\makebox(0,0){\tiny$\bullet$}}
\put(74.000,23.960){\makebox(0,0){\tiny$\bullet$}}
\put(96.250,30.080){\makebox(0,0){\tiny$\bullet$}}
\put(81.750,24.640){\makebox(0,0){\tiny$\bullet$}}
\put(78.000,18.520){\makebox(0,0){\tiny$\bullet$}}
\put(85.500,19.200){\makebox(0,0){\tiny$\bullet$}}
\put(82.000,13.080){\makebox(0,0){\tiny$\bullet$}}
\put(89.000,13.760){\makebox(0,0){\tiny$\bullet$}}
\put(86.000,7.640){\makebox(0,0){\tiny$\bullet$}}
\put(92.000,8.320){\makebox(0,0){\tiny$\bullet$}}
\put(90.000,2.200){\makebox(0,0){\tiny$\bullet$}}
\put(94.000,2.200){\makebox(0,0){\tiny$\bullet$}}
\put(110.750,24.640){\makebox(0,0){\tiny$\bullet$}}
\put(101.500,19.200){\makebox(0,0){\tiny$\bullet$}}
\put(98.000,13.080){\makebox(0,0){\tiny$\bullet$}}
\put(105.000,13.760){\makebox(0,0){\tiny$\bullet$}}
\put(102.000,7.640){\makebox(0,0){\tiny$\bullet$}}
\put(108.000,8.320){\makebox(0,0){\tiny$\bullet$}}
\put(106.000,2.200){\makebox(0,0){\tiny$\bullet$}}
\put(110.000,2.200){\makebox(0,0){\tiny$\bullet$}}
\put(120.000,19.200){\makebox(0,0){\tiny$\bullet$}}
\put(116.000,13.760){\makebox(0,0){\tiny$\bullet$}}
\put(114.000,7.640){\makebox(0,0){\tiny$\bullet$}}
\put(118.000,7.640){\makebox(0,0){\tiny$\bullet$}}
\put(124.000,13.760){\makebox(0,0){\tiny$\bullet$}}
\put(122.000,7.640){\makebox(0,0){\tiny$\bullet$}}
\put(126.000,7.640){\makebox(0,0){\tiny$\bullet$}}
\put(8.000,43.680){\makebox(0,0)[b]{\small $F$: 4 leaves}}
\put(29.000,43.680){\makebox(0,0)[b]{\small $A$: 13 leaves}}
\put(84.125,43.680){\makebox(0,0)[b]{\small $B$: 15 leaves}}
\end{picture}%
}
\caption{The $32$-leaf tree $S$. Edge lengths in the drawing are arbitrary; every positive weighting has dimension six.}
\label{fig:original}
\end{figure}

Six coordinates suffice. After removing the attachment edge, each arm
has at most $15$ leaves, so the result of
\cite[p.~281]{BrighamChartrandDuttonZhang2005} gives four internal
orientations. Reuse these coordinates on all three arms and add two
radial coordinates with distinct two-bit words for the arms. Every
cross-arm path is directed in one of the two new coordinates.

Suppose a five-cover existed. By saturation, the three interfaces
partition $Q_5$ tightly, and every descendant interface is unate.
The region for $F$ is a square by Lemma~\ref{lem:balanced-face}.
Each coordinate divides $Q_5$ into two equal halves, so the square
determines how many points the other two regions must contribute in
total to each half. Their branch shapes restrict those contributions.
For a set $U$ define its coordinate minority counts
\[
 c_i(U)=\min\{|\{x\in U:x_i=0\}|,|\{x\in U:x_i=1\}|\}.
\]
Reflections preserve these counts; permutations only reorder them.
For a down-set they count the ones in each coordinate.

\begin{lemma}\label{lem:spectra}
In a tight five-coordinate realization of the branches in
\eqref{eq:original}, the $13$-region has either minority counts
$(0,5,5,6,6)$, or all its minority counts are at least two, their sum
is $20$, and at most one equals six. The $15$-region has one of
\[
 (0,7,7,7,7),\qquad(2,5,5,6,7),\qquad(2,5,6,6,6).
\]
\end{lemma}

We prove the lemma below by following the two branch hierarchies.
First we deduce the obstruction at their common junction.

\begin{proof}[Proof that $D(S)=6$]
Normalize the square's three fixed bits to zero. On a free axis of the
square, write $a,b$ for the minority counts of the $13$- and $15$-regions.
Their actual one-counts sum to $14$, so $b=a+1$. On a fixed axis they
sum to $16$, which requires
\begin{equation}\label{eq:budgets}
 a+b=12,\qquad b=a+3,\qquad\text{or}\qquad a=b+1.
\end{equation}
A zero count in the $15$-region would require a count one in the
$13$-region. A zero in the $13$-region requires a count one or three
in the $15$-region. The listed counts exclude both possibilities.
Thus their minority sums are $20$ and $25$.

The square's two free axes contribute two to $\sum_i(b_i-a_i)$, leaving
three on its fixed axes. By \eqref{eq:budgets}, a fixed-axis difference
belongs to $\{-1,0,2,3\}$; difference zero requires $(a,b)=(6,6)$,
and difference two requires $(a,b)=(5,7)$. The only three-element
multisets from this list summing to three are
$\{3,0,0\}$ and $\{2,2,-1\}$. The first needs two counts six in the
$13$-region; the second needs two counts seven in the $15$-region.
Lemma~\ref{lem:spectra} excludes both. No five-cover exists, and the
six-coordinate construction above gives $D(S)=6$.
\end{proof}

The square forces a relation between the coordinate counts of the
two larger regions. Their branch shapes exclude every way to satisfy it.

It remains to establish the restrictions on those counts. We use the
following factorization when one child region is a face.

Coordinates outside
the indicated supports are understood to be fixed at zero.

\begin{lemma}[Face-child factorization]\label{lem:factor}
Suppose a down-set $U$ is the disjoint union of a face $Q_I$ through zero
and a nonempty unate set $C$. Let $J$ be the active coordinates of $C$
outside $I$. Then, for a nonempty down-set $D\subseteq Q_I$,
\[
 C=D\times(Q_J\setminus\{0\}),\qquad
 |C|=(2^{|J|}-1)|D|.
\]
If instead the face child has more points than its complement, it
necessarily contains zero after the parent has been made downward closed.
\end{lemma}
\begin{proof}
An outside active coordinate has its singleton in $C$, while zero belongs
to the face. Unateness therefore forces $C$ upward in each such coordinate.
Downward closure of $U$ gives the reverse propagation whenever the outside
word remains nonzero. All nonzero outside fibers are consequently the same
down-set $D$. For the last assertion, a face missing zero has a fixed
positive bit; lowering it gives a disjoint translate of the whole face
inside the parent, too large to lie in its complement.
\end{proof}

\begin{proof}[Proof of Lemma~\ref{lem:spectra}]
Identify a cube word with its support,
writing $ab$ for $\{a,b\}$ and $0$ for the empty support. Every nonempty
unate set is connected. A unate two-set is an edge and a unate three-set
is a square with one corner removed.

\smallskip\noindent\emph{Small face unions.}
Two disjoint faces of dimensions $k>l$ with unate union form a cube-isometric
copy of $I_{2^k+2^l}$. Make the union downward closed. The larger face
contains zero by Lemma~\ref{lem:factor}, so it is $Q_I$. The smaller
face has a minimum support containing an outside axis $j$. Downward
closure forces that minimum to be $\{j\}$ and all its free axes to lie
in $I$. The union is $Q_I\cup(j+Q_J)$ with $J\subset I$.
In particular, $(E,F)$ gives $I_6$, two squares sharing an edge.

The eight-leaf branch $H=(E,(E,F))$ has exactly three possible unate
shapes: a three-face, or the down-set consisting of zero, four singletons,
and the three pairs of a path or a star on those four axes.
Indeed, a weight-three point exhausts the set with its lower face.
Otherwise every square contains zero. The six-child supplies three
singletons and two pairs sharing an axis; the remaining edge adds a
fourth singleton and a pair incident to it. Each possible shape is a
union of squares and has minimum induced degree at least two.

\smallskip\noindent\emph{The thirteen-leaf branch.}
Normalize the $13$-region $U$ downward. Its eight-child $H$ is disjoint
from the three-child $P$ and the final cherry $E$.
If $U$ contains no three-face, $H$ consists of zero, four singletons
and three pairs. The disjoint connected three-child must use the fifth
singleton and two pairs incident to it. The final cherry would then
consist of two weight-two points, which are not adjacent. Thus $U$
contains a three-face.

If $U$ uses only four axes, its complement in that four-face is an
up-set of size three, hence a punctured square. This gives the counts
$(0,5,5,6,6)$.

Suppose $U$ has exactly one three-face $Q_{abc}$ and uses all five axes.
Its remaining points are the two outside singletons and three pairs.
A square not containing zero lies entirely in $Q_{abc}$, so $H$ must
contain zero. Each outside singleton needs an incident pair: otherwise
its only neighbor is zero, preventing it from lying either in $H$
(minimum degree two) or in either other connected child disjoint from
zero. The coordinate-count sum is therefore $12+2+6=20$, and every
count is at least two. Each added pair uses an outside axis, since all
pairs on $a,b,c$ are already in the face. Only three such pairs are
available, so at most one of $a,b,c$ can acquire two extra incidences.
Outside-axis counts are at most four. Hence at most one count is six.

Finally suppose there are two three-faces and five active axes. The
faces share two axes, since otherwise their union exceeds $13$ points.
Their union is
$R=Q_{ab}\times\{0,c,d\}$, of size $12$, and the last point is the
singleton $e$. This singleton cannot lie in $H$. If it lies in the final
cherry, that cherry is $\{0,e\}$; the eleven-child $R\setminus\{0\}$
has opposite boundary directions along $c$ at $0\to c$ and $d\to cd$,
so is not unate. Otherwise the three-child is $\{0,e,x\}$ for some
$x\in\{a,b,c,d\}$. The eleven-child contains $e$; its boundary edges
force downward polarity on the other axes, and its other points force
downward polarity on $e$. The removed final cherry is consequently an
upper-ideal edge of $R$. The only choices are
\[
 \{ac,abc\},\ \{bc,abc\},\ \{ad,abd\},\ \{bd,abd\}.
\]
By symmetry remove $\{bc,abc\}$. The eight-child is
$R\setminus\{0,x,bc,abc\}$. For $x=a,b$, its $c$-boundaries at
$0\to c$ and $d\to cd$ oppose. For $x=c$, use the $d$-boundaries
at $0\to d$ and $ac\to acd$; for $x=d$, use the $c$-boundaries at
$0\to c$ and $b\to bc$. None is unate. Three different three-faces
already require at least $14$ points. This proves the $13$-region claim.

\smallskip\noindent\emph{The fifteen-leaf branch.}
First the eight-child $(C_4,F)$ is a three-face. Normalize its union
downward. If the square $F$ misses zero, its lower translate exhausts
the complementary four-set. Otherwise Lemma~\ref{lem:factor} gives
$4=(2^{|J|}-1)|D|$ with $D$ contained in a square, forcing $|J|=1$
and $|D|=4$. At the next fork the eight-face has a five-point unate
complement. It contains zero in the normalized thirteen-child, and
the same factorization forces one outside axis. A down-set of size
five in $Q_3$ is a square with a pendant vertex, so this thirteen-child
is a cube-isometric copy of $I_{13}$.

Now normalize the final fifteen-set downward. Its designated eight-face
contains zero, as its complement has only seven points. The thirteen-child
is that face together with a layer in a fourth direction $j$, carrying
a five-set $D$ in $Q_3$. A fifth coordinate $k$ is fixed at zero.
The set $D$ is a square with a pendant vertex. Its downward hull has
at most seven points, so the square is the lower square in its third
direction. If $v$ denotes the pendant vertex's two free bits, that
hull has size $4+2^{|v|}$; hence $|v|=0$ or $1$.

For $|v|=1$, the missing lower pendant point must belong to the final
cherry. Its only available edge mate consistent with downward closure
is the pendant point on the other free axis. The result is a punctured
four-face, with counts $(0,7,7,7,7)$.
For $|v|=0$ the thirteen-child is already down. Its three holes in the
four-face form an upper punctured square. No two-hole edge can be added
downward: the upper hole requires both lower ones. The final cherry
cannot cross the fifth coordinate, since its upper endpoint would force
another missing lower point. It therefore lies wholly in the $k=1$
layer and has old-coordinate set $\{0,e_i\}$. The old counts
$5,5,6,6$ have one entry increased by one, and the new count is two.
This gives $(2,5,5,6,7)$ or $(2,5,6,6,6)$, completing the proof.
\end{proof}

\section{A dimension formula for three-branch trees}\label{sec:grouped}

The preceding example fails because its branch regions cannot partition
the cube. We now determine exactly when such a partition exists for a
different family: equality in the leaf bound holds precisely when its
two larger arms have equal leaf counts.

Call a $2^r$-leaf branch an \emph{rigid core} if it has an
internal $r$-cover and every tight interface in any ambient cube is an
$r$-face. Perfect branches are examples by
Lemma~\ref{lem:balanced-face}.
Call an $n$-leaf branch \emph{canonical} if it realizes $I_n$ in its
minimum cube dimension. Padding then gives the same realization in
every larger dimension.

\begin{theorem}\label{thm:grouped}
Let $p=2^r$, $c=2^s$, with $0\le s<r$, and let
\[
 0<x,y<p,\qquad x+y=2p-c.
\]
Take rigid cores $F,K_A,K_B$ with $c,p,p$ leaves and canonical
branches $X,Y$ with $x,y$ leaves. Let $G$ have three arms
\[
 F,\qquad (K_A,X),\qquad(K_B,Y).
\]
Then $G$ has $4p$ leaves and
\[
 D(G)=\begin{cases}r+2,&x=y,\\r+3,&x\ne y.\end{cases}
\]
The necessity of $x=y$ for an $(r+2)$-cover holds for arbitrary $X,Y$,
without the canonical assumption.
\end{theorem}

At the leaf bound, the two large regions must lie in opposite facets.
The small branch is a face filling their remaining holes, so those
holes must have equal sizes. The next lemma proves the first step: each
large region is confined to a facet, a face of codimension one.

\begin{lemma}\label{lem:facets}
In an $(r+2)$-cover of the junction in Theorem~\ref{thm:grouped}, each
large arm's interface lies in a facet of $Q_{r+2}$.
\end{lemma}

\begin{proof}[Proof of Theorem~\ref{thm:grouped}]
Each large region has more than $p$ points and lies in a $2p$-point
facet. Their containing facets cannot coincide. If they crossed, their
intersection would have $p$ points, whereas their two holes have total
size
\[
 (p-x)+(p-y)=c<p.
\]
The holes cannot cover the intersection, contradicting disjointness.
The facets are therefore opposite. Their complement is the $c$-face
$C$, with positive slice sizes $p-x,p-y$. A face crossing a coordinate
has equal slices, so $x=y$.

Conversely, for $x=y$ realize each large branch as $I_{p+x}$ in
$Q_{r+1}$: the core occupies a full $p$-face and the smaller canonical
child occupies $I_x$ in the opposite face. Interleave the two large
regions in even and odd positions in $Q_{r+2}$. The remaining terminal
$c$-segment is a face, occupied by the small core. The interface rule
gives an $(r+2)$-cover.

For arbitrary $x,y$, each large branch has the same $(r+1)$-coordinate
construction. Reuse these internal coordinates and add two radial
coordinates distinguishing the three arms. This gives $r+3$ coordinates.
The lower bound and the necessity just proved give the exact formula.
\end{proof}

\begin{proof}[Proof of Lemma~\ref{lem:facets}]
Saturation gives a partition into regions $C,A,B$ of sizes $c,p+x,p+y$.
The small region and the two core regions are faces. Normalize $A$
downward. Its $p$-face contains zero because $x<p$. By
Lemma~\ref{lem:factor},
\[
 A=(Q_I\times\{0\})\ \cup\ (D\times(Q_J\setminus\{0\})),
 \qquad |I|=r,\qquad x=(2^{|J|}-1)|D|.
\]
There are only two axes outside $I$. One active outside axis places
$A$ in a facet. Suppose both are active, put $d=|D|=x/3<p/3$, and write
\[
 Q_{r+2}\setminus A=(Q_r\setminus D)\times L,
 \qquad L=\{01,10,11\}.
\]
The other large region is this set minus the $c$-face $C$.
In each outside direction, the missing $00$ fiber gives $p-d>c$
boundary witnesses forcing polarity toward one. Removing $c$ points
cannot destroy all of them. Hence $B$ is upward in both outside axes.

For each inside point, the removed face's outside slice must be a lower
ideal of $L$. A face in $L$ cannot vary both axes, and an edge or the
singleton $11$ would require a third vertex below it in $L$. Thus,
after swapping the outside axes,
\[
 C=H_0\times\{01\}
\]
for a $c$-face $H_0\subseteq Q_r\setminus D$.
The $p$-point core face of $B$ must have exactly one outside axis free:
two would include $00$, and none would include all inside points,
including $D$. Its inside part is an $(r-1)$-face $H$ avoiding the
nonempty down-set $D$, so $H=\{j=1\}$ for an inside axis $j$ and
$D\subseteq\{j=0\}$. Write $W=Q_{r-1}$ and $E_0=W\setminus D$;
both $D$ and $E_0$ are nonempty. Let $K$ be the smaller-child region
obtained by removing this core from $B$. It too must be unate.

If the core's outside edge is $\{10,11\}$, a high-$j$ point in the
$01$ fiber left outside $H_0$ forces the first outside polarity of $K$
toward zero. But its low-$j$, $10$ fiber contains $E_0$ and forces
that polarity toward one. The only exception is $H_0=H$, with $c=p/2$.
Then $B$ itself has opposing $j$-boundaries: $D$ supplies them at $10$,
and $E_0$ at $01$.

If the core's outside edge is $\{01,11\}$, disjointness forces $H_0$
into the low-$j$ half. The full high-$j$, $10$ fiber of $K$ forces its
second outside polarity toward zero. Hence no low-$j$, $01$ point may
remain, giving $H_0=E_0$ and $c=p/2-d$. Now $x=3d$ and
$x+y=2p-c$ imply $y=p/2+2c$. The strict inequalities $x,y<p$ give
\[
 p/6<c<p/4,
\]
which no power of two $c$ can satisfy. Both cases are impossible.
Thus $A$ lies in a facet, and the same argument applies to $B$.
\end{proof}

To obtain explicit families, we use the following canonical branches.

\begin{lemma}\label{lem:ranktwo}
Every branch of Horton--Strahler rank at most two is canonical. Here a
leaf has rank zero; a fork of unequal ranks takes their maximum, and
a fork of two equal ranks has that rank plus one.
\end{lemma}
\begin{proof}
A rank-at-most-one branch is a comb: at every fork at least one child
is a leaf. Such a branch realizes every unate set of the right size.
To see this, normalize the set downward, remove a maximal point for
the leaf child, and proceed with the remaining down-set. Choose the
parent corner to recover the whole set.

We also need a carving fact: for $b\le n\le2^m$, a reflected copy of
$I_b$ lies in $I_n$ with unate complement. Induct on $m$, putting
$q=2^{m-1}$. If $n\le q$, work in the lower face. If $b\le q<n$,
reflect $I_b$ into the terminal part of the lower $q$-face. The remaining
two slices are $I_{q-b}$ and $I_{n-q}$; they are nested initial segments
and hence have compatible polarity in the separating coordinate.
If $q<b\le n$, carve $I_{b-q}$ out of $I_{n-q}$ in the upper face
inductively and extend the same reflection over the full lower face.
The complement lies wholly in the upper face and is unate.

At a rank-at-most-two fork one child is a comb and the other still has
rank at most two. Realize the latter canonically by induction, carve
its initial segment out of the desired parent initial segment, and
let the comb realize the unate remainder. Their union is the desired
down-set, so the parent edge with word zero adds no points.
\end{proof}

The positive holes $u=p-x$ and $v=p-y$ sum to the power of two $c$.
Unequal holes first occur as $1+3=4$, requiring $p\ge8$ and hence
$4p\ge32$. For every $k\ge5$, choose
\[
 p=2^{k-2},\qquad c=4,\qquad x=p-1,\qquad y=p-3,
\]
take perfect cores and comb siblings. Theorem~\ref{thm:grouped} gives
exactly $2^k$ leaves and dimension $k+1$. These trees are distinct
from the original tree $S$.

For canonical regions, the unequal case also follows from the fit
criterion of Soloway, Triplett and Zhao
(\href{https://arxiv.org/abs/2501.05827}{arXiv:2501.05827}, Proposition~4.17). The necessity argument
above applies to every realizable interface: Lemma~\ref{lem:facets}
excludes noncanonical shapes as well.

\section{Trees with at most 31 leaves}\label{sec:threshold}

We now prove that the construction in Section~\ref{sec:original} is
smallest. The argument places the branches at a centroid in a cube,
leaving one vertex unused when there are $31$ leaves.

\subsection{Small branch packings}
Say that $a,b$ \emph{fit} if cube-isometric copies of $I_a,I_b$ partition
$I_{a+b}$. We use the following elementary placements.
A singleton fits any size; equal and consecutive sizes fit by placing
the two copies in even and odd positions. A power of two $q$ fits every
$b\le2q$: for $b\le q$ use opposite $q$-faces; for $b=q+r$ place the
full $q$-part of $I_b$, the other $q$-face, and its $r$-remainder in
three successive $q$-blocks. Moving the $q$-bit of $I_b$ to the next
coordinate implements the last placement.

If three initial-segment copies partition $Q_m$, any chosen pair fits.
Indeed, move the third copy to $I_c$ and reflect the cube; its complement
becomes $I_{2^m-c}$. Conversely, an exact fit and the reflected terminal
segment give a three-part partition. Taking a product with $Q_1$ doubles
all three sizes. These observations also appear in
Soloway, Triplett and Zhao (\href{https://arxiv.org/abs/2501.05827}{arXiv:2501.05827}, Observation~1.10 and Corollary~4.6).

\begin{lemma}\label{lem:q4}
Every positive three-part size partition of $Q_4$ is attainable by
cube-isometric initial segments. Consequently every pair with
$a+b\le15$ fits.
\end{lemma}
\begin{proof}
Suppose the corresponding assertion holds for $Q_k$, of size $q$.
For a triple summing to $2q$ with largest part $q+r$, first handle $r=0$
by splitting one face into two reflected initial segments. For $r>0$,
partition a $q$-face into sizes $a,b,r$; adjoin the opposite full face
to the $r$-part. This gives the required $q+r$ initial-segment copy.
Thus only triples with largest part smaller than half the cube remain.
In $Q_2$ the triple is $(1,1,2)$. In $Q_3$ the only new triple is
$(2,3,3)$, which uses equal sizes. In $Q_4$ the new sorted triples are
\[
 (2,7,7),\ (3,6,7),\ (4,5,7),\ (4,6,6),\ (5,5,6).
\]
They use, respectively, an equal pair, a consecutive pair, the power-of-two
pair $(4,5)$, and two equal pairs. Three-part symmetry finishes each case.
Append the positive third size $16-a-b$ to obtain the last assertion.
If $I_{a+b}$ lies in a smaller face, the varying axes of both copies lie
there too, so the placement restricts to that face.
\end{proof}

\begin{lemma}\label{lem:canonical15}
Every attached binary branch with $n\le15$ leaves realizes the interface
$I_n$ in every dimension $m\ge\ceil{\log_2n}$.
\end{lemma}
\begin{proof}
Induct on the branch. A single edge realizes $\{0\}$. Fit the canonical
interfaces of its children exactly by Lemma~\ref{lem:q4}, and give the
new parent edge word zero. Its hull is still $I_n$, since $I_n$ is a
down-set. Padding supplies any additional dimensions.
\end{proof}

\begin{lemma}\label{lem:reserve}
If $a,b,c\le15$ are positive and $a+b+c=31$, then disjoint copies of
$I_a,I_b,I_c$ fit in $Q_5$.
\end{lemma}
\begin{proof}
We increase one part by one and partition the cube. If exactly one part
is odd, increase that part, divide all three sizes by two, use
Lemma~\ref{lem:q4}, and double. If all are odd and two are equal,
increase the remaining part and use the equal-pair fit.
For three distinct odd parts the only sorted possibilities are
\[
 (3,13,15),\ (5,11,15),\ (7,9,15),\ (7,11,13).
\]
In the first three increase $15$ to $16$ and use opposite four-faces.
In the last increase $7$ to $8$ and fit $(8,11)$ by the power-of-two
placement. Restrict the enlarged copy back to its original size.
\end{proof}

\begin{proof}[Proof of the affirmative part of Theorem~\ref{thm:first}]
Refine the tree to a cubic one with the same leaves. A leaf-weighted
centroid has at most half the leaves in each component of its complement.
Such a vertex exists: repeatedly move into a component containing more
than half the leaves; no edge can be crossed back, so this process ends.
For $31$ leaves its three branches have positive sizes at most $15$.
Realize their canonical interfaces by Lemma~\ref{lem:canonical15} and
place them disjointly by Lemma~\ref{lem:reserve}. The interface rule gives
a five-cover. There is no parent edge or hull operation at this central
three-way junction.

The same argument in $Q_4,Q_3,Q_2$ gives covers of size $4,3,2$ for
trees with $16,8,4$ leaves, respectively. Their centroid branches lie in
the canonical range. A path has a one-cover. A smaller tree can be
extended to the next one of $4,8,16,31$ leaves and the cover restricted:
subdivide a pendant edge and attach a new leaf repeatedly. Old leaf
paths remain paths in the enlarged tree. The leaf lower bound proves
equality throughout the claimed range.

For $32$ leaves, each centroid branch has at most $16$ leaves and
therefore has a four-cover internally. The two radial coordinates used
in Section~\ref{sec:original} then give a six-cover of the whole tree.
This completes the proof of Theorem~\ref{thm:first}.
\end{proof}

\subsection{Rooted interface volume}
At a chosen root, the \emph{root region} is the union of the leaf
cylinders. Each cylinder fixes the coordinates directing its root-to-leaf
path, with zero for an outward path and one for an inward path. A cover
with the fewest coordinates need not give the smallest root region.

\begin{proposition}\label{prop:b31}
Delete one leaf of a cherry in the balanced four-leaf arm of $S$, and
retain its former neighbor $v$ as the root of the remaining tree $S^-$.
Then $D(S^-)=5$, and every five-cover has root region $Q_5$ at $v$.
A six-cover has a root region of size $31$, the minimum in any dimension.
\end{proposition}
\begin{proof}
The sharp threshold gives $D(S^-)=5$. If a five-coordinate root region
omitted a word $w$, reattach the deleted leaf with outward word $w$.
For every old leaf cylinder, some fixed bit disagrees with $w$, so
the new leaf-to-old-leaf path is directed in that coordinate. This
would give a five-cover of $S$, a contradiction. Hence every
five-coordinate region fills $Q_5$.

For the six-coordinate construction, Lemma~\ref{lem:canonical15}
realizes $A$ as $I_{13}$ and $B$ as $I_{15}$ in $Q_4$.
Now place $B$ at $0,\ldots,14$ and $A$ at $16,\ldots,28$ in $Q_6$.
Their union $R$ is a down-set. The retained cherry of the four-leaf
arm occupies $\{32,34\}$, and its remaining singleton occupies
$\{33\}$. The set $R\cup\{32,34\}$ is also a down-set, so the
interface rule realizes these placements at the specified root.
The root union is
\[
 R\cup\{32,33,34\},
\]
which has $31$ points. Every root region contains all $31$ distinct
pendant signatures, so no cover gives a smaller region.
\end{proof}

\section{Dimension under leaf expansion}\label{sec:expansion}

Replacing each leaf of $S$ by a cherry gives a tree attaining the leaf bound.
We will show that the $32$-leaf example from
Section~\ref{sec:grouped} remains above the leaf bound under the
same operation. For a general tree, let $T[r]$ replace each leaf
by a perfect branch with $2^r$
leaves, retaining the old pendant edge as its attachment. Thus $T[0]=T$
and $T[r]$ has $2^r t$ leaves. The eventual value of $D(T[r])-r$ is
determined by the number of coordinates needed at each leaf.

A \emph{local certificate} consists of $M$ orientations and a
set $J_x\subseteq\{1,\ldots,M\}$ at each leaf, such that every pair
$x,y$ has a coordinate in $J_x\cap J_y$ directing its whole path.
Define the \emph{local load}
\[
 L(T)=\min\max_x|J_x|,
\]
where the minimum runs over all finite local certificates. Ordinary
covers give such certificates by taking every coordinate active.

\begin{theorem}[Stabilization]\label{thm:stabilization}
For every finite tree $T$,
\[
 L(T)=\min_{r\ge0}\bigl(D(T[r])-r\bigr).
\]
The sequence $D(T[r])-r$ is nonincreasing and eventually equal to
$L(T)$. Moreover
\[
 L(T[r])=L(T)+r,\qquad L(T)\ge\ceil{\log_2t}.
\]
\end{theorem}
\begin{proof}
Take a certificate with $M$ orientations and load $k$. Whenever
$M\le k+r$, pad it to $k+r$ coordinates. At leaf $x$, at least $r$
coordinates lie outside $J_x$; use them for the internal perfect-branch
cover. In each active coordinate, orient the entire new branch radially
with the old pendant sign. Paths within each new branch are covered
by its internal orientations. For leaves in different branches, a
coordinate active at both original leaves still directs their joining
path. Hence
\begin{equation}\label{eq:loadlift}
 D(T[r])\le k+r\quad\text{whenever }r\ge M-k.
\end{equation}

Conversely, consider an $M$-cover of $T[r]$. Each attached perfect branch
has an $r$-face through its attachment word in its interface by
Lemma~\ref{lem:balanced-face}. Select the point of this face opposite
its attachment word on the $r$ free axes, and a descendant leaf whose
cylinder contains that point. None of these $r$ axes directs the
attachment-to-descendant path, since a directing axis would fix the
attachment bit. Declare the complementary $M-r$ axes active at the
old leaf. A coordinate directing the path between two selected
descendants belongs to both active sets and directs the path between
the original leaves.
Thus $L(T)\le M-r$. Together with \eqref{eq:loadlift} for a minimum-load
certificate this proves attainment and eventual equality.

To cover $T[r+1]$, extend the orientations of $T[r]$ radially into
each new cherry and add one orientation directing all cherry paths.
Thus $D(T[r+1])\le D(T[r])+1$, proving monotonicity. Perfect
expansion satisfies $T[r][q]=T[r+q]$; applying the eventual equality
to both sides gives $L(T[r])=L(T)+r$.

Finally, the pendant signatures restricted to their active sets define
pairwise disjoint faces in $Q_M$, with all inactive coordinates free. Therefore
\[
 \sum_x2^{M-|J_x|}\le2^M.
\]
A load-$k$ certificate has $t\le2^k$, proving the last assertion.
\end{proof}

In particular, the excess splits as
\begin{equation}\label{eq:excess}
 D(T[r])-\ceil{\log_2(2^rt)}
 =\bigl(L(T)-\ceil{\log_2t}\bigr)
  +\bigl(D(T[r])-r-L(T)\bigr).
\end{equation}
The first term is the \emph{persistent excess}. The second is the
\emph{transient excess}: it is nonnegative, nonincreasing and eventually zero.

\subsection{Transient and persistent excess}

We construct families exhibiting each behavior. Mark a leaf in the four-leaf
arm of $S$. With $A,B,E$ as in \eqref{eq:original}, the remainder viewed
from that leaf is the attached branch
\[
 P_0=(1,(E,(A,B))).
\]
For $s\ge1$, put $P_s=(F_{s+4},P_{s-1})$, where $F_j$ is the
perfect branch with $2^j$ leaves. Let $T_s=P_s^+$ be the ordinary
tree obtained by counting the external attachment vertex as a leaf;
thus $T_0=S$. Let $G$ be the grouped tree with perfect cores of
sizes $4,8,8$ and comb siblings of sizes $5,7$.

\begin{theorem}\label{thm:expansion-families}
For every $s\ge0$, the tree $T_s$ has $2^{s+5}$ leaves and satisfies
\[
 D(T_s)=s+6,\qquad L(T_s)=s+5,\qquad
 D(T_s[r])=s+5+r\quad(r\ge1).
\]
In contrast, $G$ has $32$ leaves and satisfies
\[
 D(G[r])=6+r\quad(r\ge0),\qquad L(G)=6.
\]
\end{theorem}

\begin{lemma}[Root closure]\label{lem:closure}
Let $B^+$ be obtained by declaring the attachment vertex of $B$ to be
an additional leaf. An internal cover of $B$ covers $B^+$ if and only
if $U(B)\ne Q_m$.
\end{lemma}
\begin{proof}
If some radial path has no directing coordinate, its cylinder is the whole
cube. Otherwise every cylinder fixes a bit of the attachment word $a$,
so the opposite word $\bar a$ belongs to none of them. Thus the interface
is proper exactly when all attachment-to-leaf paths are covered.
\end{proof}

We also use the following consequence of the same argument. If an
interface of $B$ is a proper subset of a $k$-face, then $D(B^+)\le k$.
Indeed, all pendant signatures agree outside that face, so a coordinate
directing a leaf-pair path must be one of its free axes. A missing point
of the face disagrees with each leaf cylinder on a free axis, which
therefore directs the corresponding attachment-to-leaf path. Restricting
to the free axes covers $B^+$.

\begin{proof}[Proof of Theorem~\ref{thm:expansion-families}]
\emph{Ordinary dimension.}
The assertion for $T_0=S$ was proved in Section~\ref{sec:original}.
We first show that every interface of $P_1=(F_5,P_0)$ has at least
$64$ points, in any ambient dimension. The leaf count gives $63$.
If equality held, the perfect child's interface would be a five-face
and that of $P_0$ a unate $31$-set. Normalize their union downward.
Lemma~\ref{lem:factor} gives
\[
 31=(2^{|J|}-1)|D|.
\]
Since $31$ is prime, either $|J|=1$ and $|D|=31$, or $|J|=5$ and
$|D|=1$. In either case the interface of $P_0$ is a proper subset of
a five-face. The preceding restriction argument would give $D(S)\le5$,
a contradiction.

A full six-face is attainable. The internal tree of $P_0$ has $31$
leaves and hence a five-cover by Section~\ref{sec:threshold}. Extend
it across the attachment edge. Its interface must fill $Q_5$, since
otherwise root closure would again give a five-cover of $S$.
Place this interface and that of $F_5$ in opposite five-faces.
Inductively, for $s\ge2$, the two children of $P_s$ each require at
least $2^{s+4}$ interface points, and opposite $(s+4)$-faces attain
these bounds. Consequently $P_s$ has minimum interface size
$2^{s+5}$, attained by a full $(s+5)$-face. Root closure rules out
$s+5$ coordinates; one fixed-coordinate padding makes that face
proper and gives $D(T_s)=s+6$.

\emph{Expansion.}
Section~\ref{sec:first-expansion} constructs a six-cover of $S[1]$.
Repeating the cherry construction and applying the leaf bound give
$D(S[r])=5+r$ for every $r\ge1$.
Suppose $T_{s-1}[r]$ has an $(s+r+4)$-cover. Subdivide the connector
between its marked external perfect branch and $P_{s-1}[r]$, and
attach a new perfect branch with $2^{s+r+4}$ leaves at the new vertex.
This is exactly $T_s[r]$. Preserve the old orientations, assigning the
same direction to both parts of the subdivided edge, and use the same
$s+r+4$ coordinates to cover the new perfect branch internally.
One additional coordinate directed toward the junction throughout the
old component and away throughout the new branch covers all cross pairs.
This proves $D(T_s[r])\le s+r+5$; its leaf count gives equality.
The stabilization theorem yields $L(T_s)=s+5$.

\emph{Persistent excess.}
Under perfect expansion, every canonical region is replaced by its
product with $Q_r$, using the new coordinates internally and the old
ones radially. In binary order $I_n\times Q_r=I_{2^rn}$, and perfect
cores remain perfect. Apply Theorem~\ref{thm:grouped} to $G[r]$ with
all four parameters multiplied by $2^r$. Its unequal large branches
give $D(G[r])=6+r$. Stabilization gives $L(G)=6$.
\end{proof}

\subsection{The first expansion of the original tree}\label{sec:first-expansion}

We construct a six-cover of $S[1]$. Before expansion, the maximal
perfect terminal branches of $S$, read from left to right in
\eqref{eq:original}, have sizes
\[
 (4,2,1,2,2,2,4,2,1,1,1,2,1,1,2,4).
\]
Call these $M_0,\ldots,M_{15}$. After doubling every leaf, allocate
the faces in Table~\ref{tab:faces}, with coordinates ordered
$(a,x,y,z,b,c)$. A star denotes a free coordinate. A branch with
$2^h$ old leaves receives an $(h+1)$-face.

\begin{table}[ht]
\centering
\caption{Faces for the expanded perfect branches.}
\label{tab:faces}
\begin{tabular}{cr@{\qquad}cr}
Branch & Face & Branch & Face\\\hline
$M_0$ & \texttt{010***} & $M_8$ & \texttt{0011*0}\\
$M_1$ & \texttt{0111**} & $M_9$ & \texttt{*01001}\\
$M_2$ & \texttt{*01011} & $M_{10}$ & \texttt{0011*1}\\
$M_3$ & \texttt{*010*0} & $M_{11}$ & \texttt{1011**}\\
$M_4$ & \texttt{110**0} & $M_{12}$ & \texttt{*00001}\\
$M_5$ & \texttt{1111**} & $M_{13}$ & \texttt{*00011}\\
$M_6$ & \texttt{*110**} & $M_{14}$ & \texttt{*000*0}\\
$M_7$ & \texttt{110**1} & $M_{15}$ & \texttt{*001**}
\end{tabular}
\end{table}

These faces partition $Q_6$. For each of the thirteen remaining forks,
Table~\ref{tab:corners} specifies its children and a corner
toward which their union is closed. Within a perfect branch, use the
standard successive splitting into opposite subfaces.

\begin{table}[ht]
\centering
\caption{The remaining forks and their closure corners.}
\label{tab:corners}
\begin{tabular}{llll}
Branch & Children & Corner & Maximal supports after reflection\\\hline
$A_3$ & $M_2,M_3$ & \texttt{101010} & $ab,ac$\\
$A_6$ & $M_5,M_6$ & \texttt{111011} & $abc,zbc$\\
$A_8$ & $M_4,A_6$ & \texttt{111000} & $yzb,abc,zbc$\\
$A_{11}$ & $A_3,A_8$ & \texttt{111010} & $axb,yzb,axc,abc,zbc$\\
$A_{13}$ & $M_1,A_{11}$ & \texttt{111010} & $axb,yzb,axc,azbc$\\
$B_3$ & $M_{10},M_{11}$ & \texttt{101101} & $ab,bc$\\
$B_4$ & $M_9,B_3$ & \texttt{101101} & $az,ab,bc$\\
$B_5$ & $M_8,B_4$ & \texttt{001101} & $az,abc$\\
$C_3$ & $M_{13},M_{14}$ & \texttt{000010} & $ab,ac$\\
$C_4$ & $M_{12},C_3$ & \texttt{000001} & $abc$\\
$B_8$ & $C_4,M_{15}$ & \texttt{000001} & $azbc$\\
$B_{13}$ & $B_5,B_8$ & \texttt{000101} & $ayz,aybc,azbc$\\
$B_{15}$ & $M_7,B_{13}$ & \texttt{100101} & $ayz,xzb,aybc,azbc$
\end{tabular}
\end{table}

The last column gives a short check of each closure identity: reflect
the union by the listed corner and take the down-set generated by the
displayed supports. For example, $ab,ac$ denotes the union of the two
lower squares on those axes. Direct substitution of the face words
gives these identities. Thus the interface rule realizes each union
without enlarging it. At the central vertex the three regions
$M_0,A_{13},B_{15}$ partition $Q_6$, giving a six-cover of $S[1]$.

\section{A universal upper bound}\label{sec:universal}

For an arbitrary tree, we allow each branch more interface points
than its leaf count. We bound the resulting dimension by controlling
how this extra space accumulates when branches are joined.

For positive integers $a,b$, let $K(a,b)$ be the least $c$ such that
disjoint cube-isometric copies of $I_a,I_b$ lie in $I_c$, in its minimum
ambient cube. Thus $K(a,b)\ge a+b$. Define branch capacities by
\[
 w(1)=1,\qquad w((A,B))=K(w(A),w(B)).
\]
The interface rules give the following bound.

\begin{lemma}\label{lem:capacity}
A branch $B$ admits an internal cover with interface contained in
$I_{w(B)}$ in every dimension $m\ge\ceil{\log_2w(B)}$.
\end{lemma}
\begin{proof}
Induct on the branch, use a minimizing packing of the child initial
segments, and choose parent word zero. The union's downward hull
remains in the down-set $I_{w(B)}$.
\end{proof}

The estimate follows from a potential that is subadditive under joins.

\begin{lemma}\label{lem:potential}
For $\lambda=363/200$, there is a strictly increasing function
$\Phi:(0,\infty)\to(0,\infty)$ such that
\[
 \Phi(1)=1,\qquad \Phi(2x)=\lambda\Phi(x),\qquad
 \Phi(K(a,b))\le\Phi(a)+\Phi(b)
\]
for every $x>0$ and all positive integers $a,b$.
\end{lemma}

We construct $\Phi$ below. First we deduce the dimension bound.

\begin{proof}[Proof of Theorem~\ref{thm:universal}]
Take $\lambda,\Phi$ from Lemma~\ref{lem:potential}; in particular,
$\Phi(2^m)=\lambda^m$. Induction on the branch gives
$\Phi(w(B))\le n(B)$, where $n(B)$ is the branch's leaf count.
Root a cubic refinement of $T$ at a leaf; the remaining branch has
$n=t-1$ leaves. Set
\[
 m=\floor{\log_\lambda(t-1)}+1.
\]
Then $n<\lambda^m=\Phi(2^m)$, so $w(B)<2^m$. By
Lemma~\ref{lem:capacity} its interface lies in the proper prefix
$I_{w(B)}$, and Lemma~\ref{lem:closure} supplies an $m$-cover of the
whole tree. Contract the refinement and apply Proposition~\ref{prop:metric}.
Finally,
\[
 \floor{\log_\lambda(t-1)}+1
 \le\ceil{\log_\lambda t}
 =\ceil{\alpha\log_2t},\qquad
 \alpha=\frac{\log2}{\log\lambda}.
\]
\end{proof}

\subsection{The packing estimates}
For $a\le b$, the basic estimate is
\begin{equation}\label{eq:packing-ratios}
 K(a,b)\le b+a+\floor{a/2}.
\end{equation}
Several explicit placements improve it in parts of the parameter range.
To state all the bounds together, let $q$ be dyadic with $q\le b\le2q$
and put $A=a/q$, $B=b/q$. Let $C_0(A,B)$ be the minimum of the
eligible entries in Table~\ref{tab:capacities}. Each entry is an upper
bound for $K(a,b)/q$, with the exact integer cases at dyadic endpoints
handled by the power-of-two placement.

\begin{table}[ht]
\centering
\caption{Normalized packing capacities for $0<A\le B$, $1\le B\le2$.}
\label{tab:capacities}
\begin{tabular}{ll}
Capacity & Condition\\\hline
$B+3A/2,\ 2B,\ 2+A$ & always\\
$B+1$ & $A\le1$\\
$2+A/2$ & $B+A/2\le2$\\
$2+B/2$ & $A+B/2\le2$\\
$A+3/2$ & $1/2\le A\le1,\ B\le3/2$\\
$5/2$ & $A\le1,\ B\ge3/2,\ A+B\le5/2$\\
$A+3/2$ & $A\ge1,\ B\le3/2$\\
$B+3/2$ & $1\le A\le3/2,\ B\ge3/2$\\
$2,\ A+3B/2-1/2$ & $A+B\le2$
\end{tabular}
\end{table}

\begin{proof}[Proof of the packing estimates]
All placements
use coordinate permutations, coordinate reflections and fixed-coordinate
padding. Throughout $a\le b$, $q$ is the largest power of two at most
$b$, $b=q+r$, and $N=2q$. When $r=0$ the power-of-two construction
fits exactly, so assume $0<r<q$ when splitting a nontrivial residual.

\smallskip\noindent\emph{The basic estimate.}
We prove \eqref{eq:packing-ratios} by strong induction on the larger argument. Suppose $a+b\le N$.
If $r+a+\floor{a/2}\ge q$, the full $N$-cube suffices: put $I_b$
first and the reflected $I_a$ last. Otherwise apply induction to
$(a,r)$, whose two arguments are smaller than $q$, to get a packing
of capacity
\[
 c\le a+r+\floor{\min(a,r)/2}<q.
\]
Place it in the upper face and adjoin the full lower $q$-face to the
$r$-copy. The isometry of that copy extends over the full lower face,
so their union is a moved $I_b$. The total capacity $q+c$ satisfies
\eqref{eq:packing-ratios}.

Now suppose $a+b>N$ and put $s=N-b<a$. If $s<\ceil{a/2}$, put $I_b$
in the base $N$-face and $I_a$ above it. The capacity $N+a$ has waste
$s\le\floor{a/2}$. If $s\ge\ceil{a/2}$, move the least significant
bit of $I_a$ to the new high coordinate. Its two slices have sizes
$\ceil{a/2}$ and $\floor{a/2}$. Reflect $I_b$ inside the base face;
its first $s$ positions are free for the lower slice. The capacity
$N+\floor{a/2}$ has waste $s-\ceil{a/2}\le\floor{a/2}$.
This proves \eqref{eq:packing-ratios}.

\smallskip\noindent\emph{The remaining placements.}
Putting $I_a\subseteq I_b$ in the two parity slices gives $K\le2b$.
Separate $N$-faces give $K\le N+a$. If $a\le q$, put the full
$q$-part of $I_b$ first, $I_a$ in the next $q$-block, and the residual
$r$-part of $I_b$ in the third. This gives $K\le b+q$.

The least-significant-bit split just used also gives
\[
\begin{array}{ll}
 K\le N+\floor{a/2},&b+\ceil{a/2}\le N,\\
 K\le N+\floor{b/2},&a+\ceil{b/2}\le N.
\end{array}
\]
Because the other quantities are integers, the corresponding real
inequalities with $a/2$ or $b/2$ imply these ceiling conditions.
If $q/2<a\le q$ and $b+q/2\le N$, split $I_a$ on its $q/2$ bit
instead. Its base slice is a full $q/2$-face and its upper slice has
size $a-q/2$, giving $K\le N+a-q/2$. At $a=q/2$ the power-of-two
placement supplies the limiting bound directly.

For $a+b\le2q$ the residual construction and \eqref{eq:packing-ratios} give
\begin{equation}\label{eq:small-total}
 K(a,b)\le\min\left(2q,q+a+r+\tfrac12\min(a,r)\right).
\end{equation}
Replacing $\min(a,r)$ by $r$ also gives the
table entry $A+3B/2-1/2$.

For dyadic $q\ge2$ and $q\le z\le2q$, the pair $(3q/2,z)$ fits
by Soloway, Triplett and Zhao (\href{https://arxiv.org/abs/2501.05827}{arXiv:2501.05827}, Proposition~4.8).

Three-part symmetry and restriction of a larger segment give
\[
\begin{array}{ll}
 K\le5q/2,&a\le q,\ b\ge3q/2,\ a+b\le5q/2,\\
 K\le a+3q/2,&q<a\le b\le3q/2,\\
 K\le b+3q/2,&q<a\le3q/2<b<2q.
\end{array}
\]
For the first, fit $(b,3q/2)$ and use its complementary part of size
$5q/2-b\ge a$. For the others, enlarge the indicated smaller segment
to $3q/2$ and restrict the exact fit. Boundary cases follow from the
same placements or the exact power-of-two case. These prove the remaining table entries.

\end{proof}

One more residual step sharpens the table. If $A+B\le2$, put
$x=\min(A,B-1)$ and $y=\max(A,B-1)$. For dyadic $h$ with
$h\le y\le2h$, add the candidate
\begin{equation}\label{eq:recursive-capacity}
 1+hC_0(x/h,y/h).
\end{equation}
Take the minimum of this and the original table; call it $C(A,B)$.
When a residual argument is zero, use its exact placement instead.
To obtain \eqref{eq:recursive-capacity}, pack the two
residual pieces in the upper $q$-face and adjoin the full lower face
to the $b-q$ piece. If that packing exceeds its face, the full $2q$
cube is available. This proves
$K(a,b)\le qC(A,B)$ whenever the arguments arise from integers.

\subsection{A dyadic potential}
\begin{proof}[Proof of Lemma~\ref{lem:potential}]
Define $\phi$ on $[1,2]$ by linear interpolation
at $u_i=1+i/16$, using the following values of $10^7\phi(u_i)$ in order:
\begin{equation}\label{eq:knots}
\begin{gathered}
10000000,10201575,10841867,11528563,11927866,12528467,\\
13155523,13621714,14262133,14507897,15090325,15676375,\\
16211110,16711040,17102809,17743228,18150000.
\end{gathered}
\end{equation}
For $x=2^ku$, $1\le u\le2$, define
\[
 \Phi(x)=\lambda^k\phi(u),\qquad\Phi(0)=0.
\]
The endpoint identity $\phi(2)=\lambda\phi(1)$ makes this well-defined.
It is continuous and strictly increasing, and
\begin{equation}\label{eq:dyadic}
 \Phi(2x)=\lambda\Phi(x),\qquad\Phi(2^k)=\lambda^k.
\end{equation}

It remains to prove the packing inequality. Order the arguments and
normalize as above. Dyadic scaling reduces
the claim to
\begin{equation}\label{eq:normalized-potential}
 \Phi(C(A,B))\le\Phi(A)+\Phi(B),
 \qquad0<A\le B,\quad1\le B\le2.
\end{equation}
The knot data give
\[
 M=\max\phi'=\frac{85837}{78125},\qquad
 \rho=\min_{1\le u\le2}\frac{\phi(u)}u=\frac{363}{400}.
\]
The ratio on each affine interval has its extrema at its endpoints.
All slopes of $\Phi$ to the right of $1$ are at most $M$, since
successive dyadic scalings multiply slopes by $\lambda/2<1$.
For $0<A\le2^{-7}$, the first packing estimate yields
\[
 \Phi(C)-\Phi(B)\le\tfrac32MA,
 \qquad \frac{\Phi(A)}A\ge\rho(2/\lambda)^7\ge\tfrac32M.
\]
The last comparison follows directly from the displayed values of
$M,\rho$ and $\lambda$, and covers every $A\le2^{-7}$.

On $2^{-7}\le A\le B$, $1\le B\le2$, subdivide first at every
boundary of eligibility and equality of the affine capacities in
Table~\ref{tab:capacities}. For the residual step subdivide also at
$A=B-1$, at dyadic residual scales $h=2^j$ ($-8\le j\le0$), and
at the corresponding boundaries and equalities for the residual
placements. Subdivide further at the preimages of the interpolation
knots under $A,B,C$. The residual
$\Phi(C)-\Phi(A)-\Phi(B)$ is affine on each resulting cell, so it
suffices to check its vertices.

The accompanying rational-arithmetic checker verifies coverage by exact
polygon-area identities at every subdivision. It evaluates the residual
at $6{,}944$ distinct triples $(A,B,C)$; every value is negative.
The complete data and the independently implemented vertex check are
available with the \href{https://github.com/loganchal/pythia/tree/dd77d548f050923cc33f3017f1eddaed2e9da6d6/Research/tree_recovery_20261026}{proof sources};
the accompanying ancillary files reproduce the calculation using only
the placements proved above.
On each closed cell, the checker uses that cell's affine capacity,
including at boundary points where further placements become eligible.
Continuity therefore covers every boundary as well. Together with the
tail, these finite rational checks prove
\eqref{eq:normalized-potential}. Monotonicity and the established
integer packing upper bounds prove the lemma.
\end{proof}

\section{Closing remarks}\label{sec:closing}

The leaf bound is exact through $31$ leaves and first fails at $32$.
In the grouped family, equality holds precisely when the two larger
branches have equal leaf counts. Under leaf expansion, the excess can
disappear or persist indefinitely, and in every case it eventually
stabilizes. Our universal bound improves the logarithmic coefficient
to $1.162832$, but does not determine the largest possible excess.

Is there an absolute constant $C$ such that every $t$-leaf tree satisfies
\[
 D(T)\le\ceil{\log_2t}+C?
\]
The grouped families require $C\ge1$. Even an upper bound with
coefficient one and a lower-order additive term would improve the
universal estimate proved here.

One may also ask how large $L(T)-\ceil{\log_2t}$ can be, and how
long $D(T[r])-r$ can take to stabilize.

\section*{Statements and Declarations}

\subsection*{Funding}
This research did not receive any specific grant from funding agencies in the
public, commercial, or not-for-profit sectors.

\subsection*{Competing interests}
The author declares that he has no known competing financial interests or personal
relationships that could have appeared to influence the work reported in this paper.

\subsection*{Data availability}
The Lean proofs are available at \url{https://doi.org/10.5281/zenodo.23664556}.

\subsection*{Use of generative AI}
OpenAI's Astra assisted with exploring proofs, preparing the Lean formalisation, and
in the article preparation process. The author takes full responsibility for the article.

\end{document}